\documentclass [a4paper,abstracton]{scrartcl}

\usepackage[latin1]{inputenc}
\usepackage[T1]{fontenc}

\usepackage{lmodern} 

\usepackage{lscape}
\RequirePackage{amsmath}                       
\RequirePackage{mathtools}
\RequirePackage{amssymb}                        
\RequirePackage{latexsym}                       
\RequirePackage{stmaryrd}                       

\RequirePackage[all]{xy}                       
\UseTips                                        
\newdir^{c}{{}*!/-3pt/\dir^{(}}
\newdir_{c}{{}*!/-3pt/\dir_{(}}
\newdir^{ (}{{}*!/-3pt/\dir^{(}}
\newdir_{ (}{{}*!/-3pt/\dir_{(}}

\RequirePackage{xspace}                        
\RequirePackage{calc}                          
\RequirePackage{ifthen}                         

\RequirePackage{mathrsfs}

\RequirePackage[dvipsnames,usenames]{color} 

\RequirePackage{hyperref}
\hypersetup{
bookmarks,
bookmarksdepth=3,
bookmarksopen,
bookmarksnumbered,
pdfstartview=FitH,
colorlinks,backref,hyperindex,
linkcolor=RawSienna,
anchorcolor=BurntOrange,
citecolor=OliveGreen,
filecolor=BlueViolet,
menucolor=Yellow,
urlcolor=OliveGreen
}

   [2009/11/06 v0.1 finetuning autoref and theorem setup]

\RequirePackage{hyperref}

\makeatletter
\@ifdefinable\equationname{\let\equationname\equationautorefname}
\def\equationautorefname~#1\@empty\@empty\null{(#1\@empty\@empty\null)}
\@ifdefinable\AMSname{\let\AMSname\AMSautorefname}
\def\AMSautorefname~#1\@empty\@empty\null{(#1\@empty\@empty\null)}

\@ifdefinable\itemname{\let\itemname\itemautorefname}
\def\itemautorefname~#1\@empty\@empty\null{(#1\@empty\@empty\null)
}
\makeatother

\makeatletter
\renewcommand{\theenumi}{\alph{enumi}}

\renewcommand{\theenumii}{\roman{enumii}}

\renewcommand{\p@enumii}{\theenumi$\m@th\vert$}

\renewcommand{\p@enumiii}{\theenumi.\theenumii.}

\renewcommand{\labelitemi}{$\m@th\circ$}
\renewcommand{\labelitemii}{$\m@th\diamond$}
\renewcommand{\labelitemiii}{$\m@th\star$}
\renewcommand{\labelitemiv}{$\m@th\cdot$}
\makeatother

\RequirePackage{amsthm}
\RequirePackage{aliascnt}
\newcommand{\basetheorem}[3]{
    \newtheorem{#1}{#2}[#3]
    \newtheorem*{#1*}{#2}
    \expandafter\def\csname #1autorefname\endcsname{#2}
}
\newcommand{\maketheorem}[3]{
    \newaliascnt{#1}{#3}
    \newtheorem{#1}[#1]{#2}
    \aliascntresetthe{#1}
    \expandafter\def\csname #1autorefname\endcsname{#2}
    \newtheorem*{#1*}{#2}
}

\theoremstyle{plain} 

\basetheorem{theorem}{Theorem}{section}

\maketheorem{proposition}{Proposition}{theorem}
\maketheorem{corollary}{Corollary}{theorem}
\maketheorem{lemma}{Lemma}{theorem}
\maketheorem{conjecture}{Conjecture}{theorem}

\theoremstyle{definition}  

\maketheorem{definition}{Definition}{theorem}
\maketheorem{defprop}{Definition-Proposition}{theorem}

\theoremstyle{remark}   

\maketheorem{remark}{Remark}{theorem}
\maketheorem{remarks}{Remarks}{theorem}
\maketheorem{example}{Example}{theorem}
\maketheorem{examples}{Examples}{theorem}
\maketheorem{question}{Question}{theorem}

\usepackage{enumerate}

\newcommand{\Spec}{\operatorname{Spec}}

\newcommand{\RSHom}{\underline{\operatorname{RHom}}}
\newcommand{\SHom}{\underline{\operatorname{Hom}}}
\newcommand{\RHom}{\operatorname{RHom}}
\newcommand{\Hom}{\operatorname{Hom}}
\newcommand{\derotimes}{\overset{\LL}{\otimes}}
\newcommand{\CO}{\mathcal{O}}
\newcommand{\CF}{\mathcal{F}}
\newcommand{\CG}{\mathcal{G}}
\newcommand{\CI}{\mathcal{I}}
\newcommand{\LL}{\mathbb{L}}
\newcommand{\LX}{\overline{X}}
\newcommand{\coh}{\operatorname{coh}}
\newcommand{\id}{\operatorname{id}}
\newcommand{\ad}{\operatorname{ad}}
\newcommand{\tr}{\operatorname{tr}}
\newcommand{\qc}{\text{qc}}
\newcommand{\pr}{\operatorname{pr}}
\newcommand{\proj}{\operatorname{proj}}
\newcommand{\bc}{\operatorname{bc}}

\newcommand{\usc}[1][m]{\underline{\phantom{#1}}}

\renewcommand{\phi}{\varphi}

\renewcommand{\theta}{\vartheta}

\renewcommand{\epsilon}{\varepsilon}

\renewcommand{\to}[1][]{\xrightarrow{\ #1\ }}

\newcommand{\into}[1][]{\xhookrightarrow{\ #1\ }}

    \reversemarginpar
    \makeatletter
    \advance\marginparwidth by \ta@bcor
    \makeatother
    \newcounter{themargin}
    \def\q?#1{\textcolor{Mahogany}{\textbf{???$^{\text{\arabic{themargin}}}$}}{\marginpar{\footnotesize\color{Mahogany}\fbox{\parbox{\marginparwidth}{\textbf{? --- \arabic{themargin} ---}\addtocounter{themargin}{1}\\ #1}}} \immediate\write16{}%
    \immediate\write16{Warning: There was still a question mark . . . }%
    \immediate\write16{}}}

\usepackage{pdfpages}
\title{Grothendieck duality for non-proper morphisms}
\author{Tobias Schedlmeier}
\date{\vspace{-5ex}}
\begin{document}
\maketitle

\begin{abstract}
\noindent 
We generalize the adjunction between the functors $Rf_*$ and $f^!$ of derived categories of quasi-coherent sheaves for proper morphisms $f\colon X \to Y$ of Noetherian schemes to the following situation: Let $f$ be a finite type morphism and let $Z' \subseteq X$ and $Z \subseteq Y$ be closed subsets such that $f$ restricts to a proper morphism $f'\colon Z'\to Z$ of $f$. Then the functor $Rf_*$ is left adjoint to $R\Gamma_{Z'}f^!$ when considered as functors between complexes supported on $Z'$ or $Z$. 
\end{abstract}

\section*{Introduction}

Grothendieck's generalization of Serre duality is formulated in terms of adjoint functors. For a \emph{proper} morphism $f\colon X \to Y$ of Noetherian schemes of finite dimension it consists of the following quasi-isomorphism in the derived category of quasi-coherent sheaves on $Y$:
\begin{align} \label{GDuality}
	Rf_*\RSHom_{\CO_X}^{\bullet}(\CF^{\bullet},f^!\CG^{\bullet}) \overset{\sim}{\longrightarrow} \RSHom_{\CO_Y}^{\bullet}(Rf_*\CF^{\bullet},\CG^{\bullet})
\end{align}
for a bounded above complex $\CF^{\bullet}$ and a bounded below complex $\CG^{\bullet}$, both having coherent cohomology sheaves (\cite[VII, 3.4(c)]{HartshorneRD}). This isomorphism is known as \emph{coherent duality}. Taking $0$-th cohomology of the global sections functor, this implies that $Rf_*$ is left adjoint to the twisted inverse image functor $f^!$. The classical Serre duality for a coherent sheaf $\CF$ on a projective Cohen-Macaulay variety $X$ can be written in the form 
\[
	\Hom_k(H^i(X,\CF),k) \cong H^{n-i}(X,\SHom_{\CO_X}(\CF,\omega)),
\]
for instance. It is obtained by applying the above quasi-isomorphism \autoref{GDuality} to the structure morphism $X \to \Spec k$. 

In this paper we are concerned with morphisms $f$ which are proper only over closed subsets $Z'$ and $Z$ of $X$ and $Y$, which appear as the supports of the considered objects. Moreover, we just require that the cohomology sheaves of these objects are quasi-coherent. The main result is the following generalization of Grothendieck-Serre duality involving the functors $Rf_*$ and $R\Gamma_{Z'}f^!$, where $R\Gamma_{Z'}$ denotes the derived local cohomology functor:   
\begin{theorem*} 
Let $f\colon X \to Y$ be a separated and finite type morphism of Noetherian schemes and let $i\colon Z \to Y$ and $i'\colon Z' \to X$ be closed immersions with a proper morphism $f'\colon Z' \to Z$ such that the diagram
\[
	\xymatrix{
		Z' \ar[r]^-{i'} \ar[d]^-{f'} & X \ar[d]^-f \\
		Z \ar[r]^-i & Y
	}
\]
commutes. Then there is a natural transformation $\tr_f\colon Rf_*R\Gamma_{Z'}f^! \to \id$ such that, for all $\CF^{\bullet} \in D_{\qc}^-(\CO_X)_{Z'}$ and $\CG^{\bullet} \in D_{\qc}^+(\CO_Y)_Z$, the composition 
\begin{align*}
	\xymatrix{
		Rf_* \RSHom_{\CO_X}^{\bullet}(\CF^{\bullet},R\Gamma_{Z'}f^!\CG^{\bullet}) \ar[r] &  \RSHom_{\CO_Y}^{\bullet}(Rf_* \CF^{\bullet},Rf_*R\Gamma_{Z'}f^! \CG^{\bullet}) \ar[d]^{\tr_f} \\
		& \RSHom_{\CO_Y}^{\bullet}(Rf_* \CF^{\bullet}, \CG^{\bullet}),}
\end{align*}
where the first arrow is the natural map, is an isomorphism. Here $D_{\qc}^-(\CO_X)_{Z'}$ (and $D_{\qc}^+(\CO_Y)_Z$) denote the full subcategories of bounded above (and bounded below) complexes of the derived category of quasi-coherent sheaves on $X$ (and on $Y$) whose cohomology sheaves are supported on $Z'$ (and on $Z$). As a consequence, taking global sections, the functor $Rf_*$ is left adjoint to the functor $R\Gamma_{Z'}f^!$ between these categories. 
\end{theorem*}   
For the proof we employ Nagata compactification for the morphism $f$, which yields a factorization $X \overset{j}{\longrightarrow} \LX \overset{\overline{f}}{\longrightarrow} Y$ into an open immersion $j$ and a proper morphism $\overline{f}$. The key step is then to define the map $\tr_f$ in this more general situation by using the trace for the proper morphism $\overline{f}$. This idea has its origin in \cite{Ruelling}, where Chatzistamatiou and R\"ulling consider morphisms which are proper along a family of supports. In particular, when working with residual complexes, $\tr_f$ is even a morphism of complexes.

\section*{Acknowledgments} 

I cordially thank the supervisor of my PhD thesis, Manuel Blickle, for his excellent guidance and various inspiring conversations. I also thank Kay R\"ulling, who explained to us how the trace map underlying our adjunction could be constructed. Moreover, I thank Gebhard B\"ockle  and Axel St\"abler for many useful comments. The author was partially supported by SFB / Transregio 45 Bonn-Essen-Mainz financed by Deutsche Forschungsgemeinschaft.

\section*{Notation and conventions}

All schemes we consider are assumed to be Noetherian. In particular, all schemes and all scheme morphisms are \emph{concentrated}, i.e.\ quasi-compact and quasi-separated. For a scheme $X$, we let $D_{\qc}^*(X)$ or $D_{\coh}^*(X)$ with $* \in \{+,-,\text{b} \}$ denote the derived category of $\CO_X$-modules with quasi-coherent or coherent cohomology. Here $*=+$ or $*=-$ or $*=b$ means that we require that the cohomology sheaves are bounded below or bounded above or bounded in both directions.

\section{Local cohomology}

First, let us provide some basic facts about local cohomology which will be needed due to working with objects supported on closed subsets. Unless otherwise stated, let $i\colon Z \to X$ be a closed immersion of Noetherian schemes.    
\begin{definition} \label{deflocalcohomology}
The \emph{local cohomology functor} $R\Gamma_Z\colon D_{\qc}(X) \to D_{\qc}(X)$ is the derived functor of the left exact functor 
\[
	\Gamma_Z:= \underset{n \in \mathbb N}{\varinjlim} \SHom_{\CO_X}(\CO_X/\CI^n,\usc),
\]
where $\CI$ is any sheaf of ideals defining $Z$.
\end{definition}
A reference for local cohomology in this context is \cite{LocalHomology}. For example, \autoref{deflocalcohomology} is equation (0.1) of ibid.
\begin{definition} \label{supportdefine}
We say that a complex $\CF^{\bullet}$ of $\CO_X$-modules has \emph{support in} or \emph{on $Z$} or that $\CF^{\bullet}$ \emph{is supported in} or \emph{on $Z$} if $j^*\CF^{\bullet} = 0$ in $D(X)$. We write $D(X)_Z$, $D_{\qc}(X)_Z$ etc. for the subcategory of objects of $D(X)$, $D_{\qc}(X)$ etc. whose cohomology is supported on $Z$. 
\end{definition}   
The natural inclusion $\Gamma_Z \to \id$ induces a transformation $R\Gamma_Z \to \id$. As pointed out in the proof of \cite[Lemma (0.4.2)]{LocalHomology}, one has the following triangle:
\begin{proposition}
For every $\CF^{\bullet} \in D_{\qc}(X)$, there is a fundamental distinguished triangle
\[
	R\Gamma_Z \CF^{\bullet} \longrightarrow \CF^{\bullet} \longrightarrow Rj_*j^*\CF^{\bullet} \longrightarrow R\Gamma_Z \CF^{\bullet}[1],
\]
where the second map is the natural one from the adjunction of $Rj_*$ and $j^*$. This triangle restricts to the subcategories $D_{\qc}^+(X)$ and $D_{\qc}^b(X)$ because $j^*$ is exact and $Rj_*\colon D_{\qc}^+(U) \to D_{\qc}^+(X)$ has finite cohomological amplitude. 
\end{proposition}
In particular, $R\Gamma_Z$ only depends on the closed subset $i(Z)$ and not on the scheme structure of $Z$. The fundamental triangle allows another characterization of $D_{\qc}(X)_Z$: 
\begin{corollary}
The subcategory $D_{\qc}(X)_Z$ consists of all complexes $\CF^{\bullet} \in D_{\qc}(X)$ such that the natural map $R\Gamma_Z\CF^{\bullet} \to \CF^{\bullet}$ is an isomorphism.
\end{corollary}
\begin{lemma}
If $\CI$ is an injective quasi-coherent sheaf, then also the quasi-coherent sheaf $\Gamma_Z(\CI)$ is injective.
\end{lemma}
\begin{proof}
It suffices to check the injectivity of $\Gamma_Z(\CI)$ locally. Thus the assertion follows from (\cite[Proposition 2.1.4]{Brodmann.LocCoh}).
\end{proof}
As $R\Gamma_Z \circ R\Gamma_Z \cong R\Gamma_Z$, the image of $R\Gamma_Z$ is exactly the subcategory $D_{\qc}(X)_Z$. Furthermore, the functor $R\Gamma_Z$ is right adjoint to the inclusion $D_{\qc}(X)_Z \into D_{\qc}(X)$. This is a consequence of the following proposition, see the proof of \autoref{qcohadjunction}.
\begin{proposition} \label{Gammaadj}
Let $\CG^{\bullet}$ be a complex in $D_{\qc}(X)_Z$. Then there is a functorial isomorphism
\[
	\RSHom_{\CO_X}^{\bullet}(\CG^{\bullet},R\Gamma_Z\CF^{\bullet}) \cong \RSHom_{\CO_X}^{\bullet}(\CG^{\bullet},\CF^{\bullet})
\]
for every $\CF^{\bullet} \in D_{\qc}(X)$.
\end{proposition}    
\begin{proof}
This is \cite[Lemma (0.4.2)]{LocalHomology}. We even do not have to assume that the cohomology sheaves of $\CF^{\bullet}$ and $\CG^{\bullet}$ are quasi-coherent. 
\end{proof}
Next we verify the compatibility of $R\Gamma_Z$ with the derived functors $Rf_*$, $f^!$ and $\derotimes$. Let us fix a notation for base change:
\begin{definition} \label{bchange}
Let $f\colon X \to Y$ be a separated morphism of finite type and let $u\colon Y' \to Y$ be a flat morphism. Consider the cartesian square
\[
	\xymatrix{
		X \times_Y Y' \ar[r]^-{v} \ar[d]^-{f'} & X \ar[d]^-f \\
		Y' \ar[r]^-u & Y,
	}
\]
where $v$ and $u$ are the projections. The \emph{base change morphism} $\bc\colon u^*Rf_* \overset{\sim}{\longrightarrow} Rf'_*v^*$ is the adjoint of the composition 
\[
Rf_* \xrightarrow{Rf_*\operatorname{ad}_{v}} Rf_*Rv_*v^* \overset{\sim}{\longrightarrow} Ru_*Rf'_*v^*.
\]
Here $\operatorname{ad}_{v}$ is the unit of the adjunction between $Rv_*$ and $v^*$. 
\end{definition}
The map $\bc$ is an isomorphism in several cases. For our purposes, we will need the case of flat base change: 
\begin{lemma}[\protect{\cite[Proposition 3.9.5]{LipmanGrothDual}}]
With the notation of the preceding definition, the map $\bc$ is an isomorphism if $u$ is flat. 
\end{lemma}
\begin{lemma} \label{Gammacomm}
Let $f\colon X \to Y$ be a morphism of finite type and $i\colon Z \to Y$ a closed immersion. Let $Z'$ denote the fiber product $Z \times_Y X$ regarded as a closed subset of $X$ via the projection $Z \times_Y X \to X$.  
\begin{enumerate}
\item There is a natural isomorphism of functors
\[
	Rf_*R\Gamma_{Z'} \cong R\Gamma_ZRf_*.
\]
\item If $f$ is flat, then there is a natural isomorphism of functors
\[
	f^*R\Gamma_Z \cong R\Gamma_{Z'}f^*.
\]
\end{enumerate}
\end{lemma}
\begin{proof}
First we show that $Rf_*R\Gamma_{Z'}$ is supported on $Z$. Let $u\colon U \into Y$ and $v\colon V \into X$ be the open immersions of the complements of $Z$ and $Z'$ in $Y$ and $X$. Let $f'$ denote the restriction of $f$ to $V$. We obtain a cartesian square
\[
	\xymatrix{
		V \ar[r]^-{v} \ar[d]^-{f'} & X \ar[d]^-f \\
		U \ar[r]^-u & Y.
	}
\]
Hence $u^*Rf_*R\Gamma_{Z'} \cong Rf'_*v^*R\Gamma_{Z'} = 0$. From \autoref{Gammaadj} we know that the canonical morphism $Rf_*R\Gamma_{Z'} \to Rf_*$ factors through $R\Gamma_ZRf_*$. Let $\alpha$ denote the corresponding morphism $Rf_*R\Gamma_{Z'} \to R\Gamma_ZRf_*$. 

The square
\[
	\xymatrix@C40pt{
		Ru_*u^*Rf_* \ar[r]^-{Ru_*\bc} & Ru_*Rf'_*v^* \\
		Rf_* \ar[u]^-{\ad_uRf_*} \ar[r]^-{Rf_*\ad_v} & Rf_*Rv_*v^* \ar[u]^-{\sim}
	}
\]
commutes because $Ru_*\bc \circ \ad_uRf_*$ is the adjoint of $\bc$ and hence equals the original morphism $Rf_* \to Ru_*Rf'_*v^*$. Let $\beta$ be the composition of the natural isomorphism $Rf_*Rv_*v^* \simeq Ru_*Rf'_*v^*$ with the inverse of $Ru_*\bc$. We have just seen that the right square of the diagram
\[
	\xymatrix{
		Rf_*R\Gamma_{Z'} \ar[d]^{\alpha} \ar[r] & Rf_* \ar@{=}[d] \ar[r] & Rf_*Rv_*v^* \ar[d]_{\sim}^{\beta} \\ 
		R\Gamma_ZRf_* \ar[r] & Rf_* \ar[r] & Ru_*u^*Rf_* 
	}
\]
commutes. The left square commutes by construction. As the lines are distinguished triangles, $\alpha$ is an isomorphism. This shows (a).

For (b) we proceed similarly. The cartesian square above gives rise to the isomorphism 
\[
	\xymatrix{
		f^*R\Gamma_Z \ar[d]^{\sim} \ar[r] & f^* \ar@{=}[d] \ar[r] & f^*Ru_*u^* \ar[d]_{\sim} \\ 
		R\Gamma_{Z'}f^* \ar[r] & f^* \ar[r] & Rv_*v^*f^* 
	}
\]
of distinguished triangles. 
\end{proof}
\begin{remark} \label{underGammaPullback}
With the notation of part (b) of the preceding lemma, for every quasi-coherent sheaf $\CF$, we even have a natural isomorphism
\[
	f^*\Gamma_Z\CF \cong \Gamma_{Z'}f^*\CF,
\] 
see \cite[Lemma 4.3.1]{Brodmann.LocCoh}. 
\end{remark}
\begin{lemma} \label{RGammaTensor}    
Let $\CF^{\bullet}$ and $\CG^{\bullet}$ be complexes in $D_{\qc}^b(X)$. There are natural isomorphisms
\[
	(R\Gamma_Z\CF^{\bullet}) \derotimes_{\CO_X} \CG^{\bullet} \cong \CF^{\bullet} \derotimes_{\CO_X} (R\Gamma_Z\CG^{\bullet}) \cong R\Gamma_Z(\CF^{\bullet} \derotimes_{\CO_X} \CG^{\bullet})
\]
in $D_{\qc}(X)$.
\end{lemma}
\begin{proof}
The natural map $\CF^{\bullet} \derotimes R\Gamma_Z\CG^{\bullet} \to \CF^{\bullet} \derotimes \CG^{\bullet}$ factors through $R\Gamma_Z(\CF^{\bullet} \derotimes R\Gamma_Z\CG^{\bullet})$ because 
\[
	j^*(\CF^{\bullet} \derotimes R\Gamma_Z\CG^{\bullet}) \cong j^*\CF^{\bullet} \derotimes j^*R\Gamma_Z\CG^{\bullet} \cong 0.
\]
Let $\rho$ denote the composition of the natural isomorphism 
\[
	Rj_*j^*(\CF^{\bullet} \derotimes_{\CO_X} \CG^{\bullet}) \cong Rj_*(j^*\CF^{\bullet} \derotimes_{\CO_X} j^*\CG^{\bullet})
\]
and the isomorphism from the projection formula 
\[
	Rj_*(j^*\CF^{\bullet} \derotimes_{\CO_X} j^*\CG^{\bullet}) \cong \CF^{\bullet} \derotimes_{\CO_X} Rj_*j^*\CG^{\bullet}.
\]
We obtain a morphism of distinguished triangles 
\[
	\xymatrix{
		\CF^{\bullet} \derotimes R\Gamma_Z\CG^{\bullet} \ar[d] \ar[r] & \CF^{\bullet} \derotimes \CG^{\bullet} \ar@{=}[d] \ar[r] & \CF^{\bullet} \derotimes Rj_*j^*\CG^{\bullet} \ar[d]_-{\sim}^-{\rho} \\
		R\Gamma_Z(\CF^{\bullet} \derotimes \CG^{\bullet}) \ar[r] & \CF^{\bullet} \derotimes \CG^{\bullet} \ar[r] & Rj_*j^*(\CF^{\bullet} \derotimes \CG^{\bullet}).
	}
\]
Therefore, the left vertical arrow is an isomorphism. Analogously, one shows that 
\[
	R\Gamma_Z\CF^{\bullet} \derotimes \CG^{\bullet} \cong R\Gamma_Z(\CF^{\bullet} \derotimes \CG^{\bullet}).
\]
\end{proof}
Finally, for an open immersion $j\colon X \to \LX$, we study the connection between $R\Gamma_Z$ and $R\Gamma_{\overline{Z}}$, where $\overline{Z}$ is the closure of $Z$ in $\LX$.
\begin{definition}
Let $Z$ and $Z'$ be closed subsets of a scheme $X$. We let $D_{\qc}(X)_Z^{Z'}$ denote the full subcategory of $D_{\qc}(X)_Z$ of complexes $\CF^{\bullet}$ with $R\Gamma_{Z'} \cong 0$.
\end{definition}
\begin{proposition} \label{opensupport}
Let $j\colon X \to \LX$ be an open immersion of schemes, $Z \subseteq X$ a closed subset and $\overline{Z}$ the closure of $Z$ in $\LX$. The functors $Rj_*$ and $j^*$ restrict to inverse equivalences 
\[
	\xymatrix{
     D_{\qc}(X)_Z \ar@<.5ex>[r]^-{Rj_*} & D_{\qc}(\LX)_{\overline{Z}}^{\overline{Z} \backslash X}. \ar@<.5ex>[l]^-{j^*}
   }
\] 
\end{proposition} 		
\begin{proof}
Let $u\colon U \to X$ and $u'\colon U' \to \LX$ denote the open immersions of the complements $U$ of $Z$ in $X$ and $U'$ of $\overline{Z}$ in $\LX$. Let $j'$ be the restriction of $j$ to $U$. We obtain a cartesian square
\[
	\xymatrix{
		U \ar[r]^-u \ar[d]^-{j'} & X \ar[d]^-j \\
		U' \ar[r]^-{u'} & \LX.
	}
\]
The natural isomorphism $u'^*Rj_* \overset{\bc}{\longrightarrow} Rj'_*u^*$ shows that the essential image of $D_{\qc}(X)_Z$ under $Rj_*$ is a subcategory of $D_{\qc}(\LX)_{\overline{Z}}$. The inclusion $j$ factors through the open immersions $\sigma\colon X \to U' \cup X$ and $\tau\colon U' \cup X \to \LX$. In particular, we have a natural isomorphism $Rj_* \cong R\tau_* R\sigma_*$. Since the composition
\[
	R\tau_* R\sigma_* \xrightarrow{\id \to R\tau_* \tau^*} R\tau_* \tau^* R\tau_* R\sigma_* \xrightarrow{\tau^* R\tau_* \to \id} R\tau_* R\sigma_*
\]
is the identity and the second morphism is an isomorphism, the first map is an isomorphism too. As $\tau$ is the open immersion of the complement of $\overline{Z} \backslash X$ into $\LX$, it follows from the distinguished triangle 
\[
	R\Gamma_{\overline{Z} \backslash X} \longrightarrow \id \longrightarrow R\tau_* \tau^* \longrightarrow R\Gamma_{\overline{Z} \backslash X}[1]
\]
that $R\Gamma_{\overline{Z} \backslash X} Rj_* j^* = 0$. 

The adjunction morphism $j^*Rj_* \to \id$ is always an isomorphism. It remains to show that the natural map $\id \to Rj_*j^*$ is an isomorphism. For every $\CF^{\bullet}$ in $D_{\qc}(\LX)_{\overline{Z}}^{\overline{Z} \backslash X}$, we have $\CF^{\bullet} \cong R\Gamma_{\overline{Z}} \CF^{\bullet}$ and $R\Gamma_{\overline{Z} \backslash X} \CF^{\bullet} \cong 0$. It follows that 
\begin{align*} 
	R\Gamma_{\LX \backslash X}\CF^{\bullet} &\cong R\Gamma_{\LX \backslash X}R\Gamma_{\overline{Z}}\CF^{\bullet} \\
	&\cong R\Gamma_{\overline{Z} \backslash X}\CF^{\bullet} \\
	&\cong 0.
\end{align*}
Thus the second morphism in the fundamental triangle
\[
	R\Gamma_{\LX \backslash X}\CF^{\bullet} \longrightarrow \CF^{\bullet} \longrightarrow Rj_* j^*\CF^{\bullet} \longrightarrow R\Gamma_{\LX \backslash X}\CF^{\bullet}[1]
\]
is an isomorphism.
\end{proof}
\begin{remark}
In the standard reference \cite[Corollary II.5.11]{HartshorneRD}, Hartshorne proves that for a morphism $f\colon X \to Y$ of schemes, the functor $Lf^*$ from $D_c^-(Y)$ to $D_c^-(X)$ is left adjoint to the functor $Rf_*$ from $D^+(X)$ to $D^+(Y)$. One the one hand, we can relax the coherence assumption because in the case of an open immersion, which is a flat morphism, $f^*$ is exact. On the other hand, Proposition (3.2.1) of the more recent reference \cite{LipmanGrothDual} shows this adjunction generally for ringed spaces and without any boundedness or (quasi-)coherence assumptions on the complexes.  
\end{remark}
\begin{corollary} \label{Gammajung}
If $Z$ is a closed subset of a scheme $X$ and $j\colon X \to \LX$ is an open immersion such that the image of $Z$ in $\LX$ is closed, then there is a natural isomorphism of functors
\[
	\epsilon\colon R\Gamma_Z \overset{\sim}{\longrightarrow} Rj_*R\Gamma_Zj^*.
\]
\end{corollary}
\begin{proof}
We define $\epsilon$ as the composition of the natural map $R\Gamma_Z \to Rj_*j^*R\Gamma_Z$, which is an isomorphism by \autoref{opensupport}, and the natural isomorphism $Rj_*j^*R\Gamma_Z \overset{\sim}{\longrightarrow} Rj_*R\Gamma_Zj^*$.
\end{proof} 
For example, the condition that $Z$ is also closed in $\LX$ is satisfied if $j\colon X \to \LX$ is an open immersion of $Y$-schemes and $i\colon Z \to X$ is a closed immersion of $Y$-schemes over some base scheme $Y$ such that the structural morphisms $Z \to Y$ and $\LX \to Y$ are proper, see \autoref{ExerciseHartshorne}. When constructing the generalized trace map, we will be exactly in this situation. 
\begin{lemma} \label{opennattrans}
Let $j\colon X \to \LX$ be an open immersion. Let $Z \subset X$ be a closed subset such that $j(Z)$ is closed in $\overline{Z}$. Then for $\CF^{\bullet} \in D_{\qc}^-(X)_Z$ and $\CG^{\bullet} \in D_{\qc}^+(X)$, the natural transformation
\[
	\tau\colon Rj_* \RSHom_{\CO_X}^{\bullet}(\CF^{\bullet},\CG^{\bullet}) \to \RSHom_{\CO_{\LX}}^{\bullet}(Rj_*\CF^{\bullet},Rj_*\CG^{\bullet})
\]
is a functorial isomorphism.  
\end{lemma} 
\begin{proof}
Consider the following diagram
\[
	\xymatrix{
		Rj_*\RSHom_{\CO_X}^{\bullet}(\CF^{\bullet},\CG^{\bullet}) \ar[r]^-{\sim} \ar[d]^{\tau} & Rj_*\RSHom_{\CO_X}^{\bullet}(j^*Rj_*\CF^{\bullet},\CG^{\bullet}) \ar[d]^{\tau} \\
		\RSHom_{\CO_{\LX}}^{\bullet}(Rj_*\CF^{\bullet},Rj_*\CG^{\bullet}) \ar[r]^-{\sim} \ar@{=}[dr] & \RSHom_{\CO_{\LX}}^{\bullet}(Rj_*j^*Rj_*\CF^{\bullet},Rj_*\CG^{\bullet}) \ar[d] \\
		& \RSHom_{\CO_{\LX}}^{\bullet}(Rj_*\CF^{\bullet},Rj_*\CG^{\bullet}),
	}
\]
where the horizontal arrows are induced by the counit $j^*Rj_* \to \id$ of adjunction -- these maps are isomorphisms by \autoref{Gammajung} since $\CF^{\bullet}$ is supported in $Z$ -- and the arrow to the bottom right corner stems from the unit $\id \to Rj_*j^*$ of adjunction. The upper square commutes because of the functoriality of $\tau$. The triangle on the bottom commutes because the composition of the unit and counit of an adjunction in the manner of the diagram is canonically isomorphic to the identity. Hence the whole diagram is commutative. Finally, the composition of the two vertical arrows on the right is an isomorphism (\cite[Proposition (3.2.3)]{LipmanGrothDual}. It follows that the vertical arrow on the left is an isomorphism.    
\end{proof}

\section{Generalization of the trace map} 

The adjunction between $Rf_*$ and $f^!$ for a proper morphism $f\colon X \to Y$ is based on the \emph{trace map}, which is a natural transformation of functors
\[
	\tr_f\colon Rf_*f^! \to \id.
\]
The first step of the classical way to construct the trace is to define it for residual complexes. If $Y$ is regular, the structure sheaf $\CO_Y$ is a dualizing sheaf, and hence in particular a pointwise dualizing complex. Its Cousin complex $K^{\bullet} := E^{\bullet}(\CO_Y)$, see \cite[IV.2]{HartshorneRD}, is an injective resolution of $\CO_Y$ and an example for a \emph{residual complex}. We recall the basic facts from chapter 3.2 of \cite{conrad_grothendieck_2000}.
 	
For every morphism $g\colon X \to Y$ of finite type, one can construct a functor $g^{\Delta}$ mapping residual complexes on $Y$ to residual complexes on $X$ by gluing the functors $g^{\flat}$ for finite $g$ and $g^{\sharp}$ for separated and smooth $g$. This gives rise to the \emph{twisted} or \emph{exceptional inverse image functor}:
\begin{definition}
Let $g\colon X \to Y$ be a morphism of finite type. We define the functor $g^!\colon D_{\coh}^+(Y) \to D_{\coh}^+(X)$ by
\[
	g^! = D_{g^{\Delta}K^{\bullet}} \circ Lg^* \circ D_{K^{\bullet}},
\]
where $D$ is the duality.
\end{definition}
For proper $f$, we can define a map of complexes $\tr_f(K^{\bullet})\colon f_*f^{\Delta}K^{\bullet} \to K^{\bullet}$ (\cite[VII, Theorem 2.1]{HartshorneRD}), where $f^{\Delta}$ is the functor $f^!$ for residual complexes, see \cite[VI.3.]{HartshorneRD}. With this map in hand one defines the natural transformation $\tr_f\colon Rf_*f^! \to \id$ in the category $D_{\coh}^+(Y)$ as the unique map making the diagram
\begin{align*} 
	\xymatrix{
		Rf_*f^! \ar@{=}[r] \ar@{.>}[dd]^{\tr_f}& Rf_*\RSHom_{\CO_X}^{\bullet}(Lf^* \circ D_{K^{\bullet}}(\usc),f^{\Delta}K^{\bullet}) \ar[d]^{\sim}\\
		& \RSHom_{\CO_Y}^{\bullet}(D_{K^{\bullet}}(\usc),f_*f^{\Delta}K^{\bullet}) \ar[d]^{\tr_f(K^{\bullet})} \\
		\id \ar[r]^-{\sim} & \RSHom_{\CO_Y}^{\bullet}(D_{K^{\bullet}}(\usc),K^{\bullet})
	}
\end{align*}
commutative. Here the first vertical isomorphism on the right is the natural isomorphism from the adjunction of $Rf_*$ and $Lf^*$. Note that $f^{\Delta}(K^{\bullet})$ is injective, hence $f_*f^{\Delta}(K^{\bullet})$ computes $Rf_*f^{\Delta}(K^{\bullet})$.

Instead of constructing the twisted inverse image functor $f^!$ by pasting it from special situations such as smooth and proper maps, Lipman uses a more abstract method, the Special Adjoint Functor Theorem, to obtain a right adjoint of $Rf_*$ under weak assumptions on the morphism $f$. Then he extends this result to a ``sheafified duality'', i.e.\ for $\CF^{\bullet} \in D_{\qc}(X)$, $\CG^{\bullet} \in D_{\qc}^+(Y)$ and quasi-proper $f$, a natural isomorphism
\[
	Rf_*\RSHom_{\CO_X}^{\bullet}(\CF^{\bullet},f^!\CG^{\bullet}) \to \RSHom_{\CO_Y}^{\bullet}(Rf_*\CF^{\bullet},\CG^{\bullet}). 
\]
The compatibility of the approaches of \cite{HartshorneRD} and \cite{LipmanGrothDual} is involved, as pointed out in the introduction of \cite{LipmanGrothDual}.
  
Let us recall some results of the trace for proper morphisms.
\begin{definition}[\protect{{\cite[VI. 5.]{HartshorneRD}}}]
A morphism $f\colon X \to Y$ of schemes is called \emph{residually stable} if it is flat, integral and the fibers of $f$ are Gorenstein.
\end{definition}
\begin{lemma}[\protect{{\cite[Corollary 4.4.3]{LipmanGrothDual}}}] \label{flatcommute}
Let $f\colon X \to Y$ be proper and let $g\colon Y' \to Y$ be flat. Let $f'$ and $g'$ be the projections of $X \times_Y Y'$ such that the square
\[
	\xymatrix{
		Y' \times_Y X \ar[d]^-{f'} \ar[r]^-{g'} & X \ar[d]^-f \\
		Y' \ar[r]^-g & Y
	}
\]
is cartesian. The morphism $\beta\colon g'^*f^! \overset{\sim}{\longrightarrow} f'^!g^*$, defined as the adjoint of the composition
\[
	Rf'_*g'^*f^! \xrightarrow{\bc^{-1} f^!} g^*Rf_*f^! \xrightarrow{g^*\tr_f} g^*,
\]
is an isomorphism. Here $\bc$ denotes the base change isomorphism (\autoref{bchange}).
\end{lemma}
Let us recall two compatibilities of the trace, which usually are known as ``TRA 1'' and ``TRA 4''.
\begin{lemma} \label{proptrace}
Let $f\colon X \to Y$ be a proper morphism of schemes.
\begin{enumerate}
	\item (TRA 1) If $g\colon Y \to Z$ is another proper morphism, then there is a commutative diagram
	\[
		\xymatrix{
			R(gf)_*(gf)^! \ar[r]^-{\tr_{g \circ f}} \ar[d]^{\sim} & \id \\
			Rg_*Rf_*f^!g^! \ar[r]^-{\tr_f} & Rg_*g^! \ar[u]_{\tr_g}
		}
	\]
	where the first vertical arrow is the natural isomorphism. 
	\item (TRA 4) For a flat morphism $g\colon Y' \to Y$, there is a commutative diagram
	\[
		\xymatrix@C45pt{
			g^*Rf_*f^! \ar[r]^-{g^* \tr_f} \ar[d]^{\bc}_{\sim} & g^* \\
			Rf'_*g'^*f^! \ar[r]^-{Rf'_*\beta}_-{\sim} & Rf'_*f'^!g^*, \ar[u] _{\tr_{f'}g^*}
		}
	\]
where $g'$ and $f'$ are the two projections of $X \times_Y Y'$.
\end{enumerate}
\end{lemma}
\begin{proof}
(a) is \cite[Corollary VII.3.4]{HartshorneRD}. The diagram in (b) commutes by construction of $\beta$: The composition $\tr_{f'}g^* \circ Rf'_*\beta$ is the adjoint of the adjoint of the composition $u^* \tr_f \circ \bc^{-1}$, see \autoref{flatcommute}. 
\end{proof}
\begin{remark}
Part (b) of the preceding lemma holds under the milder assumption that $f$ is of finite Tor-dimension, see \cite[Corollary 4.4.3]{LipmanGrothDual}. In this more general case one considers the left derived functors $Lf^*$ and $Lf'^*$. However, we will need the compatibility of the trace with pullback only for flat morphisms.
\end{remark} 
From now on we do not assume that $f$ is proper. We are interested in the case where $f\colon X \to Y$ is a separated morphism of finite type of Noetherian schemes and $i\colon Z \to Y$ and $i'\colon Z' \to X$ are closed immersions with a proper morphism $f'\colon Z' \to Z$ such that $f \circ i' = i \circ f'$. The compactification theorem of Nagata (\cite{Nagata}, see also \cite{Luet} for a more recent proof) states that there exists a factorization of $f$ into an open immersion $j\colon X \to \LX$ and a proper morphism $\overline{f} \colon \LX \to Y$.
\begin{lemma} \label{ExerciseHartshorne}
Let $f\colon X \to Y$ be a morphism of schemes that factors through an open immersion $j\colon X \to \LX$ followed by a proper morphism $\overline{f}\colon \LX \to Y$. Then for every closed immersion $i\colon Z \to X$ such that $f \circ i$ is proper, the composition $j \circ i$ is also a closed immersion. 
\end{lemma} 
\begin{proof}
We have to show that $j(i(Z))$ is closed in $\LX$ (which is a special case of the first part of exercise II.4.4 of \cite{Hartshorne}). By assumption the composition $\overline{f} \circ j \circ i = f \circ i$ is proper and $\overline{f}$ is proper, in particular $\overline{f}$ is separated. Hence by \cite[Corollary II.4.8]{Hartshorne}, $j \circ i$ is proper, which implies that the image $j(i(Z))$ is closed.
\end{proof}
The following generalization of the trace map stems from \cite{Ruelling}, where Chatzistamatiou and R\"ulling define a trace for morphisms which are proper not only over a single closed subset but along a family of supports. Our construction is similar to the morphism $\operatorname{Tr}_f$ from Corollary 1.7.6 of ibid. 
\begin{definition} \label{tracedef}
For a morphism $f\colon X \to Y$ of finite type and closed immersions $i\colon Z \to Y$ and $i'\colon Z' \to X$ with a proper morphism $f'\colon Z' \to Z$ such that $f \circ i' = i\circ f'$, choose a compactification, i.e.\ an open immersion $j\colon X \to \LX$ and a proper morphism $\overline{f}\colon \LX \to X$ with $\overline{f} \circ j = f$. We obtain the following commutative diagram:
\[
	\xymatrix{
		Z' \ar@^{(->}[r]^-{i'} \ar[d]^-{f'} & X \ar[r]^-j \ar[d]^-f & \LX \ar[dl]^{\overline{f}} \\
		Z \ar@^{(->}[r]^-i & Y &
	}
\]
We define the trace of $f$ as the morphism of functors 
\[
	\tr_{f,Z} = \tr_f\colon Rf_*R\Gamma_{Z'}f^! \to \id 
\]
on $D_{\qc}^+(\CO_Y)$ given by the composition 
\begin{align*} \label{trace}
	Rf_*R\Gamma_{Z'}f^! \overset{\sim}{\longrightarrow} R\overline{f}_*Rj_*R\Gamma_{Z'}j^*\overline{f}^! \xrightarrow{R\overline{f}_*\epsilon^{-1}\overline{f}^!} R\overline{f}_*R\Gamma_{Z'}\overline{f}^! \xrightarrow{R\Gamma_{Z'} \to \id}R\overline{f}_*\overline{f}^! \xrightarrow{tr_{\overline{f}}} \id,
\end{align*}
where $\epsilon$ is the isomorphism of \autoref{Gammajung} and the last morphism is the classical Grothendieck-Serre trace for the proper map $\overline{f}$. 
\end{definition} 
Because $Rf_*R\Gamma_{Z'}f^! \cong R\Gamma_ZRf_*f^!$ (\autoref{Gammacomm}), the complex $Rf_*R\Gamma_{Z'}f^!$ is supported on $Z$. By \autoref{Gammaadj}, $\tr_f$ factors through $R\Gamma_Z$, i.e.\ there is a commutative diagram
\[
	\xymatrix{
		Rf_*R\Gamma_{Z'}f^! \ar[rr]^-{\tr_f} \ar[dr]_-{\widetilde{\tr}_f} & & \id, \\
		& R\Gamma_Z \ar[ur] &
	}
\]
where $\widetilde{\tr}_f$ is induced by $\tr_f$ and the map $R\Gamma_Z \to \id$ is the natural one. We will not distinguish between $\widetilde{\tr}_f$ and $\tr_f$. For a residual complex $E^{\bullet}$, the trace defined above is a \emph{morphism of complexes} because $f^{\Delta}E^{\bullet}$ and $\overline{f}^{\Delta}E^{\bullet}$ are residual complexes and $\Gamma_Z$ preserves injectives.  

Of course we have to show that $\tr_f$ is well-defined, i.e.\ it does not depend on the choice of a compactification. The next lemma prepares the proof of this independence. 
\begin{lemma} \label{opentrace}
Let $f\colon X \to Y$ be an open immersion. Let $Z \subseteq X$ be a closed subset such that $f(Z)$ is closed in $Y$. Then for every compactification $X \xrightarrow{j} \LX \xrightarrow{\overline{f}} Y$, the map $\tr_f$ equals the inverse $Rf_*R\Gamma_Zf^* \cong R\Gamma_Z$ of the isomorphism of \autoref{Gammajung} followed by the natural morphism $R\Gamma_Z \to \id$.  
\end{lemma}
\begin{proof}
Let $\alpha\colon R\Gamma_Z \to \id$ denote the canonical morphism of functors. The claim of the lemma is the commutativity of the diagram 
\[
	\xymatrix{
		Rf_* R\Gamma_Z f^* \ar[r]^-{\sim} \ar[d]_{\sim}^{\epsilon^{-1}} & R\overline{f}_* Rj_* R\Gamma_Z j^* \overline{f}^! \ar[d]_{\sim}^{R\overline{f}_* \epsilon^{-1} \overline{f}^!} \\
		R\Gamma_Z \ar[d]^{\alpha} & R\overline{f}_* R\Gamma_Z \overline{f}^! \ar[d]^{R\overline{f}_* \alpha \overline{f}^*} \ar[l]_-{\phi}\\
		\id & R\overline{f}_* \overline{f}^!, \ar[l]_-{\tr_{\overline{f}}}
	}
\]
where $\epsilon$ is the isomorphism of \autoref{Gammajung}. The map $\phi$ can be constructed in the following way: Let $Z'$ be the closed subset $\overline{f}^{-1}(Z)$, which contains $Z$. Then define $\phi$ as the composition 
\[
	R\overline{f}_* R\Gamma_Z \overline{f}^! \to R\overline{f}_*R\Gamma_{Z'}\overline{f}^!\overset{\sim}{\longrightarrow} R\Gamma_Z R\overline{f}_*\overline{f}^! \xrightarrow{\tr_{\overline{f}}} R\Gamma_Z
\]
of canonical transformations obtained from the natural transformation $R\Gamma_Z \to R\Gamma_{Z'}$, the isomorphism of \autoref{Gammacomm} and the trace. 

The commutativity of the bottom square is easy to see. The upper square can be extracted from the following bigger diagram:
\[
	\xymatrix{
		Rf_* R\Gamma_Z f^* \ar[r]^-{\sim} \ar[d]_{\sim} & R\overline{f}_*Rj_* R\Gamma_Z j^*\overline{f}^! \ar[d]_{\sim} \ar[dr]^{\sim} & \\
		R\Gamma_Z Rf_* f^* \ar[r]^-{\sim} & R\Gamma_Z R\overline{f}_* Rj_* j^* \overline{f}^! & R\overline{f}_* R\Gamma_{Z'} Rj_*j^*\overline{f}^! \ar[l]_-{\sim} \\
		R\Gamma_Z \ar[u]_{\ad_f}^{\sim} & R\Gamma_Z R\overline{f}_* \overline{f}^! \ar[u]_{\ad_j}^{\sim} \ar[l]_-{\tr} & R\overline{f}_* R\Gamma_{Z'} \overline{f}^!. \ar[u]_{\ad_j}^{\sim} \ar[l]_-{\sim} 
	}
\]
Here $\ad_f$ and $\ad_j$ denote the units of adjunction as in the proof of \autoref{Gammacomm}. The only part whose commutativity is not obvious is the bottom left square. For this it is enough to show that the diagram 
\begin{align} \label{opendiagram}
	\xymatrix@C50pt{
		& R\overline{f}_* \overline{f}^! \ar[r]^-{\tr} \ar[d]^{\ad_f} \ar[ddl]_{\ad_j} & \id \ar[d]^{\ad_f} \\
		& Rf_* f^* R\overline{f}_* \overline{f}^! \ar[r]^-{Rf_*f^*\tr_{\overline{f}}} \ar[d]^{\bc} & Rf_* f^* \ar@{=}[dd] \\
		R\overline{f}_* Rj_* j^* \overline{f}^! \ar[d]^-{\sim} & Rf_* Rf'_* j^* \overline{f}^! \ar[l]_-{\sim} \ar[d]^{\sim} & \\
		R\overline{f}_* Rj_* f'^* f^* & Rf_* Rf'_* f'^* f^* \ar[l]_-{\sim} \ar[r]^-{Rf_*\tr_{f'}f^*} & Rf_*f^*
	}
\end{align}
commutes. Here the morphism $\bc$ is the base change morphism with respect to the cartesian square
\[
	\xymatrix{
		X' \ar[r]^j \ar[d]^-{f'} & \overline{X} \ar[d]^-{\overline{f}} \\
		X \ar[r]^f & Y.
	}
\]
The commutativity of the upper left triangle of the diagram \autoref{opendiagram} was part of the proof of \autoref{Gammacomm}. The upper square of this diagram commutes by naturality of $\ad_f$ and the commutativity of the square below follows from \autoref{flatcommute}. Finally, the bottom left square commutes by naturality of the transformation $Rf_* Rf'_* \to R\overline{f}_* Rj_*$.    
\end{proof} 
\begin{lemma} 
The map $\tr_f$ is well-defined, i.e.\ it is independent of the choice of the compactification $j\colon X \into \LX$. 
\end{lemma}
\begin{proof}
Let $j_1\colon X \to \LX_1$ and $j_2\colon X \to \LX_2$ be two open immersions with proper morphisms $f_1\colon \LX_1 \to Y$ and $f_2\colon \LX_2 \to Y$ such that $f = f_1 \circ j_1 = f_2 \circ j_2$. By considering $\LX_1 \times_Y \LX_2$ we can reduce to the case that there is a proper morphism $g\colon \LX_1 \to \LX_2$ such that $g \circ j_1 = j_2$, i.e.\ the diagram 
\[
	\xymatrix{
		& X \ar[dl]_{j_1} \ar[dr]^{j_2} \ar[dd]_/-15pt/f & \\
		\LX_1\ \ar[dr]_{f_1} \ar[rr]_/15pt/g & & \LX_2 \ar[dl]^{f_2} \\
		& Y & 
	}
\]
commutes. That $\tr_f$ is well-defined means exactly that the following diagram of functors is commutative:
\[
	\xymatrix{
		& Rf_*R\Gamma_Zf^! \ar[dl]_-{\sim} \ar[d]_-{\sim} \ar[dr]^-{\sim} & \\
		R{f_1}_*R{j_1}_*R\Gamma_Zj_1^!f_1^!  \ar[d] & R{f_2}_*Rg_*R{j_1}_*R\Gamma_Zj_1^!g^!f_2^! \ar[l]_-{\sim} \ar[r]^-{\sim} \ar[d] & R{f_2}_*{j_2}_* R\Gamma_Z j_2^!f_2^! \ar[d] \\
		R{f_1}_*R\Gamma_Zf_1^! \ar[d]	& R{f_2}_*Rg_*R\Gamma_Zg^!f_2^! \ar[l]_-{\sim} \ar[d] & R{f_2}_*R\Gamma_Zf_2^! \ar[d] \\
		R{f_1}_*f_1^! \ar[dr]_-{\tr_{f_1}} & R{f_2}_*Rg_*g^!f_2^! \ar[l]_-{\sim} \ar[r]^-{{Rf_2}_* \tr_g f_2^!} \ar[d]^{\tr_{f_2 \circ g}} & R{f_2}_*Rf_2^! \ar[dl]^-{\tr_{f_2}} \\
		& \id. & \\
	}
\]	
Here the six vertical arrows in the middle are the natural maps occurring in the definition of $\tr_f$. The only part of which the commutativity is not obvious is the bigger rectangle on the right hand side, which follows from \autoref{opentrace} after canceling $R{f_2}_*$ and $f_2^!$ from the edges of the terms.    
\end{proof}
Note that the independence of $tr_f$ of the chosen compactification implies that $\tr_f$ equals the classical trace map whenever $f$ is proper.
\begin{proposition} \label{traceres}
The map $\tr_f$ is compatible with residually stable base change: For a residually stable morphism $g\colon S \to Y$, let $f'$ and $g'$ be the projections of $S \times_Y X$. Furthermore, let $Z_S$ and $Z_S'$ be the preimages of $Z$ and $Z'$ in $S$ and $S \times_Y X$. Then the diagram
\[	
	\xymatrix@C55pt{
		g^*Rf_*R\Gamma_{Z'}f^! \ar[d]_-{\bc}^-{\sim} \ar[r]^-{g^* \tr_f} & g^* \\
		Rf'_*g'^*R\Gamma_{Z'}f^! \ar[d]^-{\sim} & \\
		Rf'_*R\Gamma_{Z_S'}g'^*f^! \ar[r]^-{Rf'_*R\Gamma_{Z_S'}\beta} & Rf'_*R\Gamma_{Z_S'}f'^!g^* \ar[uu]_-{\tr_{f'}g^*} 
	}
\]
commutes. Here $\beta$ is the isomorphism of \autoref{flatcommute}.  
\end{proposition}
\begin{proof}
First we treat the case of an open immersion $u\colon U \to Y$. Let $h\colon S \to Y$ be a residually stable morphism and let $u'$ and $h'$ denote the projections of $S \times_Y U$. Again, we let $\ad_u$, $\ad_{u'}$, $\ad_{h'}$ and $\ad_{h \circ u'}$ denote the units of adjunction. The natural isomorphisms $Ru_*Rh'_* \cong Rh_*Ru'_*$ and $h'^*u^* \cong u'^*h^*$ are compatible with the adjunction of $(u \circ h')^*$ and $R(u \circ h')_*$, i.e.\ the diagram
\[
	\xymatrix{
		\id \ar[r]^-{\ad_{h \circ u'}} \ar[d]_-{\ad_u} & Rh_*Ru'_*u'^*h^* \ar[r]^-{\sim} & Rh_*Ru'_*h'^*u^* \ar@{=}[d] \\
		Ru_*u^* \ar[r]^-{\ad_{h'}} & Ru_*Rh'_*h'^*u^* \ar[r]^-{\sim} & Rh_*Ru'_*h'^*u^* 
	}
\]
of natural maps commutes.  Passing to the adjoint maps we see that the square
\[
	\xymatrix{
	h^* \ar[r]^-{\ad_{u'}} \ar[d]_-{\ad_u} & Ru'_*u'^*h^* \\
	h^*Ru_*u^* \ar[r]^-{\bc} & Ru'_*h'^*u^* \ar[u]_-{\sim}
	}
\]	
is commutative. Applying the derived local cohomology functor and taking the inverse of the now invertible units of adjunction (\autoref{Gammajung}), we obtain the commutative diagram
\[
	\xymatrix{
		h^*Ru_*R\Gamma_{Z'}u^* \ar[r]^-{h^* \tr_u} \ar[d]_-{\bc} \ar@{.>}[dr]^-{d}& h^* \\
		Ru'_*h'^*R\Gamma_{Z'}u^* \ar[r]^-{\sim} & Ru'_*R\Gamma_{Z_U'}u'^*h^*, \ar[u]_-{\tr_{u'} h^*} 
	}
\]
where $d$ denote the composition $h^*Ru_*R\Gamma_{Z'}u^* \overset{\bc}{\longrightarrow} Ru'_*h'^*R\Gamma_{Z'}u^* \overset{\sim}{\longrightarrow} Ru'_*R\Gamma_{Z_U'}u'^*h^*$.

Now we choose a compactification $X \xrightarrow{j} \LX \xrightarrow{\overline{f}} Y$ of $f$. Then $S \times_Y X \xrightarrow{j'} S \times_Y \LX \xrightarrow{\overline{f'}} S$ is a compactification of $f'$ where $j' := \id \times j$ and $\overline{f'}$ is the projection. Let $\overline{g'}\colon S \times_Y \LX \to \LX$ denote the projection onto $\LX$. The three squares in the commutative diagram
\[
	\xymatrix{
		& S \times_Y \LX \ar[rr]^-{\overline{g'}} \ar@/^4mm/[dddl]^{\overline{f'}} & & \LX \ar@/^4mm/[dddl]^{\overline{f}} \\
		S \times_Y X \ar[ur]^{j'} \ar[rr]^/5mm/{g'} \ar[dd]_{f'} & & X \ar[ur]^j \ar[dd]_f & \\
		& & & \\
		S \ar[rr]^-g & & Y &
	}
\]	
are cartesian. It suffices to show the commutativity of 
\[
	\xymatrix{
		g^*\overline{f}_*j_*\Gamma_{Z'}j^*\overline{f}^! \ar[r]^-{\bc} \ar[d]_{\sim}^{\tr_j} & \overline{f'}_*\overline{g'}^*j_*\Gamma_{Z'}j^*\overline{f}^! \ar[r]^-{d} \ar[d]_{\sim}^{\tr_j} & \overline{f'}_*j'_*\Gamma_{Z_S'}j'^*\overline{g'}^*\overline{f}^! \ar[r]^-{\beta} \ar[d]_{\sim}^{\tr_{j'}} & \overline{f'}_*j'_*\Gamma_{Z_S'}j'^*\overline{f'}^!g^* \ar[d]_{\sim}^{\tr_{j'}} \\
		g^*\overline{f}_*\Gamma_{Z'}\overline{f}^! \ar[r]^-{\bc} \ar[d]^{\tr_{\overline{f}}} & \overline{f'}_*\overline{g'}^*\Gamma_{Z'} \overline{f}^! \ar[r]^-{\sim} & \overline{f'}_*\Gamma_{Z_S'}\overline{g'}^*\overline{f}^! \ar[r]^-{\beta} & \overline{f'}_* \Gamma_{Z_S'} \overline{f'}^! g^* \ar[d]^{\tr_{\overline{f'}}} \\
		g^* \ar[rrr]^-{\id} & & & g^*,
	}
\]
where we left out the $R$ indicating derived functors to streamline the notation. The first and the third upper square commute because of the naturality of $\tr_j$ and $\tr_{j'}$. The commutativity of the upper square in the middle is the case of an open immersion, which we have already seen. Finally, the commutativity of the bottom rectangle is the case of a proper morphism (\autoref{proptrace}). 
\end{proof}
\begin{proposition} \label{tracecomp}
Let $f\colon X \to Y$ and $g\colon Y \to S$ be separated and finite type morphisms of schemes. Assume that $i\colon Z \to S$, $i'\colon Z' \to Y$ and $i''\colon Z'' \to X$ are closed immersions with proper morphisms $f'\colon Z'' \to Z'$ and $g'\colon Z' \to Z$ making the diagram
\[
	\xymatrix{
		Z'' \ar[r]^-{i''} \ar[d]^-{f'} & X \ar[d]^-f \\
		Z' \ar[r]^-{i'} \ar[d]^-{g'} & Y \ar[d]^-g \\
		Z \ar[r]^-i & S
	}
\]
commutative. Then there is a commutative diagram:  
\[
	\xymatrix@C40pt@R30pt{
		R(g \circ f)_*R\Gamma_{Z''}(g \circ f)^! \ar[d]^{\sim} \ar[drr]^-{\tr_{g \circ f}} & & \\ 
		Rg_*Rf_*R\Gamma_{Z''}f^!g^! \ar[r]_-{Rg_* \tr_f g^!} & Rg_*R\Gamma_{Z'}g^! \ar[r]_-{\tr_g} & \id.
	}
\]	
\end{proposition}
\begin{proof}
Choose a compactification $Y \xrightarrow{j_Y} \overline{Y} \xrightarrow{\overline{g}} S$ of $g$, then choose a compactification $X \xrightarrow{j_X} \LX \xrightarrow{f'} \overline{Y}$ of the composition $j_Y \circ f$. The morphisms $f$ and $j_X$ induce a morphism $h\colon X \to Y \times_{\overline{Y}}\LX$. The projection $\pr_Y\colon Y \times_{\overline{Y}} \LX \to Y$ is proper because it is the base change of the proper morphism $f'$. The projection $\pr_{\LX}\colon X \times_{\overline{Y}}\LX \to \LX$ is a base change of $j_Y$ and hence an open immersion. We obtain the following commutative diagram:
\[
	\xymatrix{
		X \ar[d]_-f \ar[r]^-{h} & Y \times_{\overline Y} \LX \ar[dl]^-{\pr_{Y}} \ar[r]^-{\pr_{\LX}} & \LX \ar[dl]^-{f'} \\
		Y \ar[d]_-g \ar[r]_-{j_Y} & \overline{Y} \ar[dl]^-{\overline{g}} & \\
		S.
	}
\]
Because $\pr_{\LX} \circ h$ equals the open immersion $j_X$, it follows that $h$ is an open immersion too. The asserted compatibility of the trace map now follows from the compatibility of the classical trace with compositions applied to the proper morphisms $f'$ and $\overline{g}$ and using \autoref{traceres} for $f'$ and the open immersion $j_Y$. The details of the calculation are left to the reader. 
\end{proof}

\section{Adjunction for morphisms with proper support}

With his approach to the functor $f^!$ mentioned in the preceding section, Lipman proved the following version of Grothendieck duality (\cite[Corollary 4.4.2]{LipmanGrothDual}).
\begin{theorem} \label{adjunctionclassic}
	Let $f\colon X \to Y$ be a proper morphism between Noetherian schemes. For $\CF^{\bullet} \in D_{\qc}(X)$ and $\CG^{\bullet} \in D_{\qc}^+(Y)$, the composition
	\begin{align*} 
	\xymatrix{
		Rf_* \RSHom_{\CO_X}^{\bullet}(\CF^{\bullet},f^!\CG^{\bullet}) \ar[r] &  \RSHom_{\CO_Y}^{\bullet}(Rf_* \CF^{\bullet},Rf_*f^! \CG^{\bullet}) \ar[d]^{\tr_f} \\
		& \RSHom_{\CO_Y}^{\bullet}(Rf_* \CF^{\bullet}, \CG^{\bullet})}
	\end{align*}
	is an isomorphism. Here the first morphism is the canonical one and the second is the trace map. 
\end{theorem}
This generalizes the classical coherent duality (\cite[VII, 3.4(c)]{HartshorneRD}), where $\CF^{\bullet} \in D_{\coh}^-(X)$ and $\CG^{\bullet} \in D_{\coh}^+(Y)$ and $Y$ is assumed to  have a dualizing complex. In this paper we relax the properness assumption and show the following:
\begin{theorem} \label{qcohadjunction} 
	Let $f\colon X \to Y$ be a separated and finite type morphism of Noetherian schemes and let $i\colon Z \to Y$ and $i'\colon Z' \to X$ be closed immersions with a proper morphism $f'\colon Z' \to Z$ such that the diagram 
	\[
	\xymatrix{
		Z' \ar[r]^-{i'} \ar[d]^-{f'} & X \ar[d]^-f \\
		Z \ar[r]^-i & Y
	}
	\]
	commutes. Then for all $\CF^{\bullet} \in D_{\qc}^-(\CO_X)_Z$ and $\CG^{\bullet} \in D_{\qc}^+(\CO_Y)_Z$ (see \autoref{supportdefine}), the composition 
	\begin{align*}
	\xymatrix{
		Rf_* \RSHom_{\CO_X}^{\bullet}(\CF^{\bullet},R\Gamma_{Z'}f^!\CG^{\bullet}) \ar[r] &  \RSHom_{\CO_Y}^{\bullet}(Rf_* \CF^{\bullet},Rf_*R\Gamma_{Z'}f^! \CG^{\bullet}) \ar[d]^{\tr_f} \\
		& \RSHom_{\CO_Y}^{\bullet}(Rf_* \CF^{\bullet}, \CG^{\bullet})}
	\end{align*}
	is an isomorphism. Here $\tr_f\colon Rf_*R\Gamma_{Z'}f^! \to \id$ is the natural transformation of \autoref{tracedef}. In particular, taking global sections, the functor $Rf_*$ is left adjoint to the functor $R\Gamma_{Z'}f^!$. 
\end{theorem}
Note that the properness of $f'$ is equivalent to the properness of $i \circ f'$.
\begin{proof}
Consider the commutative diagram
\[
	\xymatrix{
		Rf_*\RSHom_{\CO_X}^{\bullet}(\CF^{\bullet},R\Gamma_{Z'}f^!\CG^{\bullet}) \ar[d] \ar[dr] & \\
		R\overline{f}_*\RSHom_{\CO_{\LX}}^{\bullet}(Rj_*\CF^{\bullet},Rj_*R\Gamma_{Z'}j^*\overline{f}^!\CG^{\bullet}) \ar[r] \ar[d]^{\epsilon^{-1}} & \RSHom_{\CO_Y}^{\bullet}(Rf_*\CF^{\bullet},R\overline{f}_*Rj_*R\Gamma_{Z'}j^*\overline{f}^!\CG^{\bullet}) \ar[d] \\
		R\overline{f}_*\RSHom_{\CO_{\LX}}^{\bullet}(Rj_*\CF^{\bullet},R\Gamma_{Z'}\overline{f}^!\CG^{\bullet}) \ar[r] \ar[d] & \RSHom_{\CO_Y}^{\bullet}(Rf_*\CF^{\bullet},R\overline{f}_*R\Gamma_{Z'}\overline{f}^!\CG^{\bullet}) \ar[d] \\
		R\overline{f}_*\RSHom_{\CO_{\LX}}^{\bullet}(Rj_*\CF^{\bullet},\overline{f}^!\CG^{\bullet}) \ar[dr] \ar[r] & \RSHom_{\CO_Y}^{\bullet}(Rf_*\CF^{\bullet},R\overline{f}_*\overline{f}^!\CG^{\bullet}) \ar[d] \\
		& \RSHom_{\CO_Y}^{\bullet}(Rf_*\CF^{\bullet},\CG^{\bullet})
	}
\]
of natural morphisms. The vertical arrows on the left are isomorphisms by \autoref{opennattrans}, \autoref{Gammajung} and \autoref{Gammaadj}. The diagonal map to the lower right corner is the isomorphism from the duality of \autoref{adjunctionclassic} for the proper morphism $\overline{f}$. Hence the composition of the first diagonal morphism and the vertical morphisms on the right is an isomorphism.  

Finally, for the adjunction of $Rf_*$ and $R\Gamma_{Z'}f^!$, we apply the degree zero cohomology of the right derived functor of global sections $H^0R\Gamma$ to both sides of the just proven isomorphism 
\[
	\RSHom_{\CO_Y}^{\bullet}(Rf_*\CF^{\bullet},\CG^{\bullet}) \overset{\sim}{\longrightarrow} Rf_*\RSHom_{\CO_X}^{\bullet}(\CF^{\bullet},R\Gamma_{Z'}f^!\CG^{\bullet}).
\]
Then we use the natural isomorphisms 
\begin{align*}
	H^0R\Gamma\RSHom_{\CO_Y}^{\bullet}(Rf_*\CF^{\bullet},\CG^{\bullet}) &\overset{\sim}{\longrightarrow} H^0\RHom_{\CO_Y}^{\bullet}(Rf_*\CF^{\bullet},\CG^{\bullet}) \\
	&\overset{\sim}{\longrightarrow} \Hom_{D(\CO_Y)}(Rf_*\CF^{\bullet},\CG^{\bullet})
\end{align*}
of Proposition II.5.3 and Theorem I.6.4 of \cite{HartshorneRD} and similarly for 
\[
	H^0R\Gamma Rf_*\RSHom_{\CO_X}^{\bullet}(\CF^{\bullet},R\Gamma_{Z'}f^!\CG^{\bullet}), 
\]
where we additionally use the isomorphism $R\Gamma(X,\usc) \overset{\sim}{\longrightarrow} R\Gamma(Y,Rf_*(\usc))$ of Proposition II.5.2 of ibid.
\end{proof}
We conclude with a statement which, under certain hypothesis, allows us to recover the trace $\tr_f$ by its application to the structure sheaf $\CO_Y$. 
\begin{definition} \label{EssentiallyPerfect}
For a separated morphism $f\colon X \to Y$ of finite type, a compactification $X \xrightarrow{j} \LX \xrightarrow{\overline{f}} Y$ and $\CF^{\bullet} \in D_{\qc}^+(Y)$, let 
\[
	\chi_{\CF^{\bullet}}^f\colon f^!\CO_Y \derotimes_{\CO_X} Lf^*\CF^{\bullet} \to f^!\CF^{\bullet}
\]
be the morphism $j^*\phi$, where $\phi\colon \overline{f}^!\CO_Y \derotimes_{\CO_X} L\overline{f}^*\CF^{\bullet} \to \overline{f}^!\CF^{\bullet}$ is the adjoint of the composition
\[
	R\overline{f}_*(\overline{f}^!\CO_Y \derotimes_{\CO_X} L\overline{f}^*\CF^{\bullet}) \xrightarrow{\rho} R\overline{f}_*(\overline{f}^!\CO_Y) \derotimes_{\CO_Y} \CF^{\bullet} \xrightarrow{\tr_{\overline{f}}} \CO_Y \otimes_{\CO_Y} \CF^{\bullet}.
\]
Here $\rho$ denotes the isomorphism of the projection formula. The morphism $\chi_{\CF^{\bullet}}^f$ is independent of the choice of the compactification \cite[Proposition 5.8]{Na.CompEFT}. If $\chi^f$ is an isomorphism of functors, then the morphism $f$ is called \emph{essentially perfect}. 
\end{definition}
Theorem 5.9 of \cite{Na.CompEFT} Nayak gives various characterizations of essentially perfect morphisms. For example, smooth morphisms are essentially perfect.
\begin{proposition} \label{Nayak}
Let $f$ be an essentially perfect map fulfilling the assumptions and with the notation of \autoref{qcohadjunction}. For every $\CF^{\bullet} \in D_{\qc}^+(Y)$, there is a commutative diagram
\[
	\xymatrix{
		Rf_*R\Gamma_{Z'}(f^!\CO_Y \derotimes_{\CO_X} Lf^*\CF^{\bullet}) \ar[r]^-{\chi_{\CF^{\bullet}}^f} \ar[d]^-{\rho} & Rf_*R\Gamma_{Z'}f^!\CF^{\bullet} \ar[d]^-{\tr_f} \\
		Rf_*R\Gamma_{Z'}f^!\CO_Y \derotimes_{\CO_Y} \CF^{\bullet} \ar[r]^-{\tr_f \otimes \id} & \CF^{\bullet}.
	}
\]
Here $\rho$ denotes the isomorphism of the projection formula and \autoref{RGammaTensor}.
\end{proposition}
\begin{proof}
We have to verify that the diagram
\begin{align} \label{NayakDiagram}
	\xymatrix{
		Rf_*R\Gamma_{Z'}(f^!\CO_Y \derotimes_{\CO_X} Lf^*\CF^{\bullet}) \ar[r]^-{\chi_{\CF^{\bullet}}^f} \ar[d]^-{\sim} & Rf_*R\Gamma_{Z'}f^!\CF^{\bullet} \ar[d]_-{\sim} \\
		R\overline{f}_*Rj_*R\Gamma_{Z'}(j^*\overline{f}^!\CO_Y \derotimes j^*L\overline{f}^*\CF^{\bullet}) \ar[d]^-{\rho} & R\overline{f}_*Rj_*R\Gamma_{Z'}j^*\overline{f}^!\CF^{\bullet} \ar[dd]_-{\sim}^{\epsilon^{-1}} \\
		R\overline{f}_*(Rj_*R\Gamma_{Z'}j^*\overline{f}^!\CO_Y \derotimes L\overline{f}^*\CF^{\bullet}) \ar[d]_-{\epsilon^{-1}}^-{\sim} & \\
		R\overline{f}_*(R\Gamma_{Z'}\overline{f}^!\CO_Y \derotimes L\overline{f}^*\CF^{\bullet}) \ar[d]^-{\rho} \ar[r]^-{\phi} & R\overline{f}_*R\Gamma_{Z'}\overline{f}^!\CF^{\bullet} \ar[dd]^-{\tr_{\overline{f}}} \\
		Rf_*R\Gamma_{Z'}f^!\CO_Y \derotimes_{\CO_Y} \CF^{\bullet} \ar[d]^{\tr_f \otimes \id} & \\
		\CO_Y \otimes \CF^{\bullet} \ar[r]^-{\sim} & \CF^{\bullet}
	}
\end{align}
commutes. The upper rectangle commutes because the projection formula is compatible with the unit $\id \to Rj_*j^*$ of adjunction, which we denote by $\ad_j$. More precisely, it follows from the commutativity of the diagram
\[
	\xymatrix@C40pt{
	\overline{f}^!\CO_Y \derotimes L\overline{f}^* \CF^{\bullet} \ar[r] \ar[d] & Rj_*j^*(\overline{f}^!\CO_Y \derotimes L\overline{f}^* \CF^{\bullet}) \ar[d] \\
	(Rj_*j^*\overline{f}^!\CO_Y) \derotimes L\overline{f}^* \CF^{\bullet} \ar[r] \ar[d]^-{\proj} & Rj_*j^*((Rj_*j^*\overline{f}^!\CO_Y) \derotimes L\overline{f}^* \CF^{\bullet}) \ar[d]_-{\sim} \\
	Rj_*(j^*\overline{f}^!\CO_Y \derotimes j^*L\overline{f}^* \CF^{\bullet}) & Rj_*(j^*(Rj_*j^*\overline{f}^!\CO_Y) \derotimes j^*L\overline{f}^* \CF^{\bullet}), \ar[l]_-{\widetilde{\ad}_j}
	}
\]
where the maps of the upper square stem from $\ad_j$ (this square commutes by the naturality of the unit of adjunction), where $\proj$ is the isomorphism from the projection formula and where the lower horizontal arrow is obtained from the counit of adjunction $\widetilde{\ad}_j\colon j^*Rj_* \to \id$. The lower rectangle commutes by construction of $\proj$. The composition
\[
	j^* \xrightarrow{j^* \ad_j} j^*Rj_*j^* \xrightarrow{\widetilde{\ad}_j j^*} j^*
\]
is the identity. Therefore, the composition of the vertical arrows on the right hand side and the lower horizontal arrow equals $Rj_*$ applied to the natural isomorphism 
\[
	j^*(\overline{f}^!\CO_Y \derotimes L\overline{f}^* \CF^{\bullet}) \to j^*\overline{f}^!\CO_Y \derotimes j^*L\overline{f}^* \CF^{\bullet}.
\]
The bottom rectangle of the diagram \autoref{NayakDiagram} commutes by construction of $\chi_{\CF^{\bullet}}^{\overline{f}}$.  
\end{proof}

\bibliographystyle{amsalpha}

\bibliography{Bibliography}

\end{document}